\documentclass[12pt]{amsart}

\usepackage{amsmath}
\usepackage{amssymb}
\usepackage{bm}
\usepackage{mathrsfs}
\usepackage{xcolor}
\usepackage{enumitem}
\setlist[enumerate]{parsep=0pt plus 4pt,topsep=0pt plus 4pt}
\usepackage{url}
\usepackage{mathtools}
\allowdisplaybreaks
\usepackage{epsfig}
\usepackage{xspace}

\newcommand\excise[1]{}

\newtheorem{theorem}{Theorem}
\newtheorem{thm}{Theorem}[section]
\newtheorem{lemma}[thm]{Lemma}
\newtheorem{prop}[thm]{Proposition}

\newtheorem{cor}[thm]{Corollary}

\newtheorem*{claim*}{Claim}
\theoremstyle{definition}
\newtheorem{defn}[thm]{Definition}
\newtheorem{remark}[thm]{Remark}

\newcounter{separated-sec}
\newcounter{collapse-sec}

\newcommand{\Ring}[1]{\ensuremath{\mathbb{#1}}}

\renewcommand\>{\rangle}
\newcommand\<{\langle}

\newcommand\EE{\mathbb{E}}

\newcommand\MM{\mathcal{M}}
\newcommand\NN{\Ring{N}}

\newcommand{\RR}{\Ring{R}}

\newcommand{\dd}{\mathbf{d}}
\newcommand{\dtan}{d}

\newcommand\xx{\mathbf{x}}
\newcommand\bP{\mathbf{P}}
\newcommand\cB{\mathcal{B}}
\newcommand\cC{\mathcal{C}}
\newcommand\cF{\mathcal{F}}

\newcommand\ve{\varepsilon}

\newcommand\bmu{{\bar\mu}}

\newcommand\hmu{{\widehat\mu}}

\newcommand\Gn{G_n}

\newcommand\Emu{E_\mu}

\newcommand\Fmu{F}

\newcommand\Smu{S_\bmu\MM}
\newcommand\SpM{S_p\MM}

\newcommand\Tmu{T_\bmu\MM}
\newcommand\TpM{T_p\MM}

\newcommand\bgn{\ol g_n}

\newcommand\ccdot{\;\cdot\;}

\newcommand\nablabmu{\nabla_{\hspace{-.25ex}\bmu}}

\newcommand\ol[1]{{\overline{#1}}}
\newcommand\wh[1]{{\widehat{#1}}}

\newcommand\interior[1]{{\kern0pt#1}^{\mathrm{o}}}
\newcommand\edits[1][]{}

\DeclareMathOperator*\argmin{argmin}
\DeclareMathOperator\CAT{CAT}

\DeclareMathOperator\tangCOV{\Sigma}
\DeclarePairedDelimiter\abs{\lvert}{\rvert}
\DeclarePairedDelimiter\norm{\lVert}{\rVert}

\newcommand\noheight[1]{\raisebox{0pt}[0pt][0pt]{#1}}

\definecolor{teal}{HTML}{029386}
\begin{document}

\mbox{}\vspace{-4ex}
\title[CLT for random tangent fields on stratified spaces]%
{A central limit theorem for random tangent fields on stratified spaces}
\author{Jonathan C. Mattingly}
\address{\rm Departments of Mathematics and of Statistical Sciences,
Duke University, Durham, NC 27708}
\urladdr{\url{https://scholars.duke.edu/person/jonathan.mattingly}\vspace{-1ex}}
\author{Ezra Miller}
\address{\rm Departments of Mathematics and of Statistical Sciences,
Duke University, Durham, NC 27708}
\urladdr{\url{https://scholars.duke.edu/person/ezra.miller}\vspace{-1ex}}
\author[Do Tran]{Do Tran}%
\address{\rm Georg-August Universit\"at at G\"ottingen, Germany\vspace{-1ex}}

\makeatletter
  \@namedef{subjclassname@2020}{\textup{2020} Mathematics Subject Classification}
\makeatother
%
%

\date{2 January 2025}

\begin{abstract}
Variation of empirical Fr\'echet means on a metric space with
curvature bounded above is encoded via random fields indexed by unit
tangent vectors.  A~central limit theorem shows these random tangent
fields converge to a Gaussian such~field and lays the foundation for
more traditionally formulated central limit theorems in
subsequent work.
\end{abstract}

\maketitle
\tableofcontents

\section*{Introduction}\label{s:intro}

\noindent%
The classical central limit theorem (CLT) captures the distributional
form of the leading order fluctuation of a sample mean around its
limiting value, namely the true or population mean.  After rescaling,
these fluctuations typically converge to a Gaussian distribution.
When the samples are drawn from a linear space, such as a flat vector
space, where addition of elements is well defined, both the sample and
population means are given by the classical average.  These classical
averages can also be identified as the minimizers of a least-squared
error problem involving the empirical and population distribution,
which largely explains the centrality of averages in asymptotic
problems.  While averages do not translate to nonlinear spaces, since
geometry precludes simple analogues of arithmetic averages, the
minimization formulation does translate, as it requires only a metric.
Minimizers of the \emph{Fr\'echet function}
\begin{equation}\label{eq:FrechetF}
  F_\mu(p) = \frac12\int_\MM d(p,x)^2\,\mu(dx)
\end{equation}
of the given measure~$\mu$ on a space~$\MM$, known as barycenters,
play the central role of the \emph{Fr\'echet mean} of the
distribution~$\mu$.  Furthermore, variation of the empirical
barycenter from the \mbox{population} barycenter~$\bmu$ is informed by
the geometry of~$\MM$ and is precisely the object, once rescaled, to
be described via the CLT.

In settings where the geometric structure of $\MM$ is
smooth~\cite{bhattacharya-patrangenaru2005}, spatial variation in
$\MM$ maps invertibly to the tangent space, at least locally near the
population mean~$\bmu$. This allows one to consider the limit of the
scaled fluctuations in the tangent space and then push them back down
to the base space $\MM$.
However, advances in geometric statistics have uncovered the need to
work with singular spaces~$\MM$ (see \cite{shadow-geom}, especially
sources cited in the Introduction there) that may not have a smooth
geometric structure.  These spaces typically have singular points,
whose tangent spaces are cones rather than flat, Euclidian vector
spaces as in the smooth setting.  In the absence of smoothness, it is
impossible to transform tangent data into a linear space invertibly.
%
%
Instead of taking into account all spatial variation of empirical
means around~$\bmu$, our approach here reduces the singular situation
to the linear case by treating variation around the population mean
radially---that is, ray by ray.  The result is a potent interpretation
of the CLT as convergence of real-valued random fields indexed by the
unit sphere (Theorem~\ref{t:tangent-field-CLT}).  This perspective
makes sense even when the underlying space $\MM$ is neither smooth nor
linear.  The sequel \cite{escape-vectors} to this paper uses the
results of this note on CLTs for random tangent fields to prove a more
classical-looking CLT for full spatial variation of empirical means
in~$\MM$ around~$\bmu$.

To see how this proceeds, and to highlight connections with the
classical case, begin by returning to a linear setting.  \edits{Simplify
notation by writing $F_\mu = \Fmu$ for the Fr\'echet function and $F_n
= F_{\mu_n}$ for the \emph{empirical Fr\'echet function} of the empirical measure $\mu_n=\frac1n\sum_k \delta_{X_k}$ corresponding to i.i.d.\
sample of random points~$X_1,\dots,X_n$~from~$\mu$, so}
\begin{align*}
  \notag
F_n(p)
  & = \frac 12 \int |p-x|^2 \mu_n(dx)
    = \frac 12 \int \bigl(|p|^2 - 2 p\cdot x + |x|^2\bigr) \mu_n(dx)
\\\notag
  & = \frac{|p|^2}2 - p \cdot \bmu + \!\!\int\! \frac{|x|^2}2\mu(dx)
      - \frac 1n \sum_{i=1}^n (p\cdot X_i - p\cdot \bmu)
      + \frac 1n \sum_{i=1}^n \frac{|X_i|^2}2 - \!\!\int\! \frac{|x|^2}2\mu(dx)
\\
  & = \Fmu(p) - \bgn(p) + C_n,
\end{align*}
where
$$
  \bgn(p) =\frac1n \sum_{i=1}^n (p\cdot X_i - p\cdot \bmu)
  \quad\text{and}\quad
  C_n = \frac 1{2n} \sum_{i=1}^n |X_i|^2 - \frac 12 \int |x|^2\mu(dx).
$$
Since $C_n$ is a constant that is independent of~$p$, depending only
on~$\mu_n$ and~$\mu$, minimizing~$F_n(p)$ coincides with minimizing
$\Fmu(p) - \bgn(p)$.  Hence $\sqrt n\, (\bmu_n - \bmu)$ becomes
$$
  \sqrt n \Bigl(\argmin_p F_n(p) - \argmin_p \Fmu(p)\Bigr)
  =
  \sqrt n \Bigl(\argmin_p \bigl(\Fmu(p)
    - \bgn(p)\bigr)
    - \argmin_p \Fmu(p)\Bigr).
$$
Since the setting is linear for the moment, the coordinate system can
be translated so that $\bmu$ is the origin.  Then this equation
simplifies, and the usual linear CLT becomes
\begin{align}\label{eq:recast:1}
  \sqrt n\,(\bmu_n - \bmu)
  \overset{d}=
  \lim_{n\to\infty} \sqrt n \argmin_p \bigl(\Fmu(p) - \bgn(p)\bigr).
\end{align}

This phrasing shifts perspective from a spatial view---minimize the
sum of square distances to the sample points---to an angular view:
maximize the sum of inner products with the sample points without
unduly increasing the Fr\'echet function of the population measure.
The task of a CLT becomes comparing these two
quantities~asymptotically.  Asymptotics of the right-hand side
of~\eqref{eq:recast:1} can be understood by first getting a handle on
convergence of the average functions~$\bgn$ and then passing that
perturbation through the minimization problem in~\eqref{eq:recast:1}.
Before expanding on that, the first step is to identify an appropriate
setting where $\MM$ is not linear.

As in the smooth case, any singular CLT naturally occurs tangentially
at the Fr\'echet mean~$\bmu$.  For this to occur, the space~$\MM$ must
possess a space of tangent data with enough structure.  To that end
the setting here is $\CAT(\kappa)$ spaces
(Section~\ref{s:random-tangent-field}), a rather general class of
metric spaces in which each point $p \in \MM$ has a tangent
cone~$\TpM$.  $\CAT(\kappa)$ spaces have enough structure to define an
angular pairing that serves as a
replacement for the inner product in a linear~space.

Using a $\CAT(\kappa)$ space~$\MM$ and its tangent cone~$\Tmu$ at the
Fr\'echet mean, the CLT~\eqref{eq:recast:1} is recast as an expression
for the distribution of the limit
\begin{equation}\label{eq:coneCLT}
  \lim_{n\to\infty} \sqrt n \argmin_{V \in \Tmu}
    \bigl(\Fmu(\exp_\bmu V) - \bgn(V)\bigr)
\end{equation}
on~$T_\bmu$, where now
$\bgn$ takes place in the tangent cone instead of~$\MM$ itself, so
$X_i = \log_\bmu x_i$ for~$x_i \sim \mu$:
\begin{equation}\label{eq:inner-products}
\begin{split}
\bgn(V)&= \frac 1n \sum_{i=1}^n \bigl(\<V, X_i\>_\bmu - m(V)\bigr)
\\\quad\text{and}\quad
   m(V)&= \int_\MM \<V,\log_\bmu x\>_\bmu\,\mu(dx)
	  =
	  \int_{\Tmu} \<V,U\>_\bmu\,\,\hmu(dU),
\end{split}
\end{equation}
where $\hmu = (\log_\bmu)_\sharp\mu$ is the pushforward of the
measure~$\mu$ in~$\MM$ to the cone~$\Tmu$.  The random real-valued
random field~$\bgn$, indexed by the tangent cone~$\Tmu$, and its
convergence are the principal topics of this paper.  Motivated by
these formulas, random tangent fields are defined in
Section~\ref{s:random-tangent-field}.  Crucially, even though $\Tmu$
is not a vector space when $\bmu$ is a singular point, the conical
geometry of~$T_\bmu$ (Definition~\ref{d:tangent-cone}) has enough
structure for angles and nonnegative scaling to be defined in~$\Tmu$.
This allows the definition of an angular pairing that plays the role
of inner products in~\eqref{eq:inner-products} when the setting is
singular (Definition~\ref{d:inner-product}).

Shifting to the formulation~\eqref{eq:coneCLT} has a decisive
advantage.  Since the tangent cone $\Tmu$ need not be a vector space,
classical CLTs do not apply to collections of independent random
variables taking values in~$\Tmu$.  However, classical CLTs do apply
to collections of real-valued random fields indexed by the unit sphere
$\Smu$ in~$\Tmu$, as they can be viewed as a vector space.  And $\bgn$
in~\eqref{eq:inner-products} can be viewed as a real-valued random
field indexed by~$\Smu$ which is the average of the $n$ independent
random fields $V \mapsto \<X_k, V\>_\bmu - m(V)$.  This change in
perspective from CLT as convergence of random variables to CLT as
convergence of random fields is a powerful theoretical advance that
has added value even in classical~cases; see
\cite[Remark~6.40]{escape-vectors}.

In more detail, since each of the $n$ fields averaged in~$\bgn$ is
centered and independent, an appropriate CLT for random fields yields
convergence of $\sqrt n\,\bgn$---again,~viewed as a random field---to
a Gaussian random field~$G$; this is the main
Theorem~\ref{t:tangent-field-CLT}.  This Gaussian random field has
mean zero and covariance $\EE[G(V) G(U)] = \Sigma(V,U)$, where
\begin{equation*}\label{eq:Rcov}
\begin{split}
  \Sigma: \Tmu \times \Tmu
  & \to \RR \quad\text{with}\quad
  (U,V)  \mapsto  \Sigma(U,V):=\int_{\Tmu} \<U,y\>_\bmu\<V,y\>_\bmu\, \hmu\: dy.
\end{split}
\end{equation*}

The convergence in Theorem~\ref{t:tangent-field-CLT} addresses what
happens inside the $\argmin$ from~(\ref{eq:coneCLT}), but it falls
short of relating the limit~(\ref{eq:coneCLT}) to asymptotic spatial
variation of empirical Fr\'echet means \noheight{$\lim_{n\to\infty}
\sqrt n \log_\bmu \bmu_n$} in analogy with~(\ref{eq:recast:1}).  That
is treated in subsequent work \cite{escape-vectors}, which pushes this
convergence to the tangent cone as a whole, rather than ray by~ray,
using tangential collapse \cite{tangential-collapse} and additional
theory surrounding variation of Fr\'echet means on a singular
space~$\MM$ upon perturbation of the measure~$\mu$.

\subsection*{Acknowledgements}
\noindent
DT was partially funded by DFG HU 1575/7.  The other two authors were
informed of \cite{sticky-flavors-2023}, which has related results, the
day before posting this paper.  JCM thanks the NSF RTG grant
DMS-2038056 for general support as well as Davar Khoshnevisan and
Sandra Cerrai for enlightening conversations.  The referees provided
extremely helpful, perceptive comments that substantively improved
the paper.%

\numberwithin{equation}{section}

\section{Prerequisites}\label{s:prereqs}

\noindent
To keep this paper self-contained, recall some basic notions regarding
$\CAT(\kappa)$ spaces and measures on them from the prequel
\cite{tangential-collapse}; more details, proofs, background, and
examples can be found there and in the sources it cites (notably
\cite{BBI01} and \cite{shadow-geom}).

\edits{Throughout the paper, fix a complete locally compact
$\CAT(\kappa)$ metric space~$(\MM,\dd)$.}  Many of the definitions fix
an arbitrary point $\bmu \in \MM$.  In practice, this point is always
the Fr\'echet mean of a measure~$\mu$ in this paper, but it does no
harm to call the point~$\bmu$ even without the presence of a
measure~$\mu$.

\begin{defn}\label{d:angle}\rm
The \emph{angle} between geodesics $\gamma_i\colon [0,\ve_i) \to \MM$ for
$i = 1,2$ emanating from~$\bmu$ in the $\CAT(\kappa)$ space~$(\MM,\dd)$
and parametrized by arclength is characterized by
$$
  \cos\bigl(\angle(\gamma_1,\gamma_2)\bigr)
  =
  \lim_{t,s\to 0}\frac{s^2+t^2-\dd^2(\gamma_1(s),\gamma_2(t))}{2st}.
$$
\edits{Clearly, $\angle(\gamma_1,\gamma_2) \in [0,\pi]$.} The geodesics~$\gamma_i$ are \emph{equivalent} if the angle between
them is~$0$.  \edits{The completion} $\Smu$
\edits{of} the set of equivalence classes is the \emph{space
of directions} at~$\bmu$.
\end{defn}

\begin{lemma}\label{l:angular-metric}
The space of directions is a length space whose \emph{angular
metric}~$\dd_s$~satisfies
$$
  \dd_s(V, W) = \angle(V, W)
  \text{ whenever } V, W \in \Smu
  \text{ with } \angle(V, W) < \pi
$$
\end{lemma}
\begin{proof}
This is \cite[Proposition~1.7]{shadow-geom}, which in turn is
\cite[Lemma~9.1.39]{BBI01}.
\end{proof}

\begin{defn}\label{d:tangent-cone}\rm
At any point~$\bmu$ the $\CAT(\kappa)$ space~$\MM$ has \emph{tangent
cone}
$$
  \Tmu = \Smu \times [0,\infty) / \bigl(\SpM \times \{0\}\bigr),
$$
defined in the standard way as a topological quotient of the cylinder
over~$\Smu$, whose apex
is also called~$\bmu$.  The \emph{length} $\|V\|$ of a tangent
vector~$V$ is its $[0,\infty)$ coordinate.  The \emph{unit tangent
sphere} of length~$1$ vectors is identified with~$\Smu$, the space
of~directions.
\end{defn}

\begin{defn}\label{d:inner-product}\rm
Tangent vectors $V, W \in \Tmu$ at any point~$\bmu$ have
\emph{angular pairing}
$$
  \<V,W\>_\bmu = \|V\|\|W\| \cos\bigl(\angle(V,W)\bigr).
$$
The subscript $\bmu$ is suppressed when the point~$\bmu$ is clear from
context.
The \emph{conical metric} on the tangent cone $\Tmu$,\edits{which
differs from the angular metric $d_s$ from above},~is
$$
  \dtan_\bmu(V,W)
  =
  \sqrt{\|V\|^2 + \|W\|^2 - 2\<V,W\>_\bmu}\
  \text{ for } V,W \in \Tmu.
  $$
\end{defn}

Though not a true inner product, angular paring shares enough
properties with an inner product, such as the Cauchy--Schwartz
inequality (immediate from Lemma~\ref{l:angular-metric}) and
continuity with fixed basepoint~$\bmu$ (this is
\cite[Lemma~1.21]{shadow-geom}, which derives it from
continuity
of~$\dd_s$ in Lemma~\ref{l:angular-metric}), to
replace inner products in this singular setting.

\begin{defn}\label{d:log-map}\rm
If $\bmu \in \MM$ and $v$ lies in the set $\MM' \subseteq \MM$ of
points in the $\CAT(\kappa)$ space~$(\MM,\dd)$ with unique shortest
path to~$\bmu$, then write $\gamma_v$ for the unit-speed shortest path
from~$\bmu$ to~$v$ and $V$
for its
direction at~$\bmu$.  The \emph{log~map} is defined~by
\begin{align*}
  \log_\bmu \colon \MM' & \to \TpM \quad \text{where}\quad
                 v \mapsto \dd(\bmu,v) V.
\end{align*}
\end{defn}

\begin{defn}[{\cite[Definition~2.1]{tangential-collapse}}]\label{d:localized}\rm
A measure~$\mu$ on $\CAT(\kappa)$~$\MM$ is \emph{localized}~if
\begin{itemize}
\item%
its Fr\'echet mean~$\bmu$ is unique,
%
\item%
the Fr\'echet function of~$\mu$ is locally convex in a neighborhood
of~$\bmu$, and
%
\item%
the logarithm map $\log_\bmu\colon \MM \to \Tmu$
is uniquely defined $\mu$-almost surely.
\end{itemize}
The pushforward of a localized measure~$\mu$ to~$\Tmu$ is denoted
$$
  \hmu = (\log_\bmu)_\sharp\mu.
$$
\end{defn}

\begin{defn}\label{d:directional-derivative}\rm
Given a localized measure~$\mu$ on a $\CAT(\kappa)$ space~$\MM$, the
\emph{directional derivative} of the Fr\'echet function~$F$ at~$\bmu$
is
\begin{align*}
  \nablabmu F : \Tmu & \to \RR  \quad \text{where}\quad
                  V \mapsto \frac{d}{dt} F(\exp_\bmu tV)|_{t=0}
\end{align*}
where the exponential is a constant-speed geodesic with tangent~$V$
at~$\bmu$.
\end{defn}

\edits[\vspace{-3ex}]{
\begin{remark}\label{r:nabla-well-defined}
Directional derivatives are well defined by
\cite[Corollary~2.7]{tangential-collapse}, which states that
$$
  \nablabmu F(V) = -\int_{\Tmu}\<W,V\>_\bmu\,\hmu(dW).
$$
\end{remark}
}

\begin{defn}[Stratified $\CAT(\kappa)$ space]\label{d:weakly-stratified}\rm
A complete, geodesic, locally compact $\CAT(\kappa)$ space $(\MM,\dd)$
is \emph{stratified} if it decomposes
$\MM = \bigsqcup_{j=0}^d \MM^j$
into disjoint locally closed \emph{strata}~$\MM^j$ so that for
each~$j$, the stratum~$\MM^j$ has closure
$$
  \overline{\MM^j} = \bigcup_{k\leq j}\MM^k,
$$
and for each stratum $\MM^j$, the space $(\MM^j,\dd|_{\MM^j})$ is a
smooth manifold with geodesic distance $\dd|_{\MM^j}$ that is the
restriction of~$\dd$ to~$\MM^j$.
\end{defn}

\begin{remark}\label{r:weakly-stratified}
The notion of stratified $\CAT(\kappa)$ space in
Definition~\ref{d:weakly-stratified} is weaker (i.e., more inclusive)
than the notion of smoothly stratified metric space from
\cite[Definition~3.1]{tangential-collapse}:
a stratified $\CAT(\kappa)$ space becomes a smoothly stratified metric
space if its log maps are locally invertible as in
\cite[Definition~3.1.2]{tangential-collapse}.  The only location in
this paper where the stratified hypothesis enters is
Lemma~\ref{l:d-is-finite}, with the consequence that
Theorem~\ref{t:kolmogorov-reg} assumes $\MM$ is stratified and
therefore Lemmas~\ref{l:tight} and~\ref{l:functional-CLT} do, as well.
The rest of the arguments work for arbitrary complete locally compact
$\CAT(\kappa)$ spaces.
\end{remark}

\section{Random tangent fields}\label{s:random-tangent-field}

\begin{defn}[Random tangent field]\label{d:random-tangent-field}\rm
Fix a complete probability space $(\Omega,\cF,\bP)$ and a smoothly
stratified metric space~$\MM$.  A~\emph{random tangent field} on the
unit tangent sphere~$\Smu$ at a point $\bmu \in \MM$ is a
\edits{real-valued} stochastic process~$f$ indexed
by~$\Smu$; that is,
\begin{align*}
  f\colon \Omega \times \Smu & \to \RR
\end{align*}
is a measurable map.  Typically $\omega$ is suppressed in these
notations.  A \emph{centered random tangent field} is such a field~$f$
with expectation $\EE f(V) = 0$ for all $V \in \Smu$.
\end{defn}

\begin{remark}\label{r:random-tangent-field}
A random tangent field~$f$ can be thought of---and notated---as a
collection of functions $\{f(V)\colon \Omega \to \RR \mid V \in \Smu\}$
indexed by unit tangent vectors~$V$ at~$\bmu$.  If~$x = x(\omega)$ is
an $\MM$-valued random variable $\Omega \to \MM$, and $g\colon \MM \times
\Smu \to \RR$ is any (ordinary deterministic) function, then $g$
induces a random tangent field written
$$
  (\omega,V) \mapsto g\bigl(x(\omega),V\bigr).
$$
In the interest of brevity, we typically abuse notation to write
$g(x,V)$ for this random field.  For example \edits{given an
$\MM$-valued random variable $x$, the function $g\colon (y,V)
\mapsto \<\log_\bmu y, V\>_\bmu$ induces a random tangent field
$g(x,V)$}, the centered instance of which appears in
Definition~\ref{d:tangent-field}.
\end{remark}

\begin{remark}\label{r:extend-to-Tmu}
A random tangent field~$f$ can be canonically extended to a
homogeneous stochastic process indexed by~$\Tmu$ instead of~$\Smu$,
where $f$ is \emph{homogeneous} if $f(tV) = tf(V)$ for all real~$t
\geq 0$ and~$V \in \Tmu$.
This extension is not needed in this paper, but it arises in the
sequel \cite[Section~3.2]{escape-vectors}.
\end{remark}

This paper studies asymptotic behavior of the centered random tangent
field induced by a probability measure~$\mu$ on~$\MM$.  Write $\Fmu =
F_\mu$ for the Fr\'echet function of~$\mu$ in~\eqref{eq:FrechetF}.


\begin{defn}\label{d:moments-tangent-field}\rm
For any \edits{localized} probability measure $\nu$
on~$\MM$, define the
\begin{enumerate}
\item%
\emph{tangent mean function} for \edits{$V\in \Smu$} by
$$
  V \mapsto m(\nu,V)
  :=
  \int_\MM \<\log_\bmu y, V\>_\bmu\nu(dy)
$$
\item%
\emph{tangent covariance kernel} for $(V,W) \in \Smu \times \Smu$ by
$$
  (V,W) \mapsto \tangCOV(\nu,V,W) 
  :=
  \int_\MM \bigl[\<\log_\bmu y, V\>_\bmu - m(\nu,V)\bigr]\,
           \bigl[\<\log_\bmu y, W\>_\bmu - m(\nu,W)\>_\bmu\bigr]\,\nu(dy).
$$
\end{enumerate}
Abusing notation again, these definitions extend to any $\MM$-valued
random variable $x\colon \Omega \to \MM$ with law $\nu$ by setting $m(x,V)
= m(\nu,V)$ and $\tangCOV(x,V,W) = \tangCOV(\nu,V,W)$, so
\begin{align*}
  m(x,V)
  &=
  \EE \<\log_\bmu x, V\>_\bmu
\\\text{and}\quad
  \tangCOV(x,V,W)
  &=
  \EE \Bigl[\bigl(\<\log_\bmu x, V\>_\bmu - m(x,V)\bigr)\,
            \bigl(\<\log_\bmu x, W\>_\bmu - m(x,W)\bigr)\Bigr].
\end{align*}
\end{defn}

\begin{defn}\label{d:tangent-field}\rm
For any $\MM$-valued random variable $x\colon \Omega \to \MM$, the
\emph{centered random tangent field} at~$\bmu$~is
\begin{equation*}
  \tau(x,V) 
  :=
  \<\log_\bmu x, V\>_\bmu - m(x,V)\,.
\end{equation*}
If $x$ is distributed as~$\mu$ then $\tau(V) = \tau(x,V)$ is the
\emph{centered tangent field induced by~$\mu$}.
\end{defn}

\begin{remark}\label{r:centered}
The centered random tangent field $\tau(x,^{\!} V)$ has mean zero and 
covariance $\tangCOV(x,^{\!}V,W)$ in that $\EE \tau(x,^{\!}V) = 0$ and 
$\EE \, \tau(x,^{\!}V) \tau(x,^{\!}W) = \tangCOV(x,^{\!}V,W)$
for~all~$V,W \!\in\!  \Smu$.
\end{remark}

\begin{remark}\label{r:nablamuF(V)=m(mu,V)}
Centering by subtracting the tangent mean function $m(\mu,V)$ is the
same as adding the derivative $\nablabmu F(V)$ by
\cite[Corollary~2.7]{tangential-collapse} (see
Remark~\ref{r:nabla-well-defined} for the statement).  In particular,
the tangent mean function $m(\mu,V) = -\nablabmu F(V)$ is never
strictly positive (as $\bmu$ is a minimizer of~$F$) and vanishes
precisely when $\nablabmu F(V) = 0$.  This vanishing
defines the \emph{escape cone}~$\Emu$
\cite[Definition~2.12]{tangential-collapse},
so $m(\mu,V) \leq 0$ for all~$V \in \Smu$, and $m(\mu,V) = 0$ if and
only if $V \in \Emu$ by definition.
\end{remark}

\begin{defn}[Gaussian tangent field]\label{d:gaussian-tangent-field}\rm
A \emph{Gaussian tangent field with covariance $\tangCOV$} is a
centered random tangent field $G\colon \Omega \times \Smu \to \RR$ such
that
\begin{enumerate}
\item%
the \emph{covariance} $(V,W) \mapsto \EE\,G(V)G(W)$ of~$G$
is $\tangCOV(V,W)$ for all $V,W \in \Smu$,~and
\item%
for all $V_1,\ldots, \hspace{-1pt}V_n \!\in\! \Smu$,
$\bigl(G(V_1),\ldots, G(V_n)\bigr)$ is multivariate Gaussian
\mbox{distributed}.
\end{enumerate}
The \emph{Gaussian tangent field induced by~$\mu$} is the centered
Gaussian random tangent field $G\colon \Omega \times \Smu \to \RR$ whose
covariance is $\EE G(V)G(W) = \tangCOV(\mu,V,W)$.
\end{defn}

\section{Regularity of random tangent fields}\label{s:regularity}

\noindent
The main result in this section is a Kolmogorov regularity theorem
adapted to our setting \edits{of random tangent fields}
(Theorem~\ref{t:kolmogorov-reg} in \S\ref{b:kolmogorov-reg}) and its
interpretation for Gaussians \edits{random tangent fields}
(Corollary~\ref{c:regGaussian} in~\S\ref{b:Holder}).  Presentations of
these are followed by the proof of Theorem~\ref{t:kolmogorov-reg}
(in~\S\ref{b:chaining}).

\subsection{A Kolmogorov regularity theorem}\label{b:kolmogorov-reg}

\begin{theorem}\label{t:kolmogorov-reg}
Fix a point~$\bmu$ in a stratified $\CAT(\kappa)$ space~$\MM$.  Let
$f$ be a random tangent field on the unit tangent sphere~$\Smu$ with
\begin{equation}\label{eq:ExpCont}
  \EE |f(V) -f(W)|^p \leq C_p \dtan_\bmu(W, V)^p
\end{equation}
for all $p \geq 1$ for some finite constant~$C_p$.  Then there exists
a version of the random tangent field $f$ which is almost surely a
H\"older continuous function $\Smu \to \RR$ for all H\"older exponents
in~$(0,1)$.
\end{theorem}
The proof, at the end of Section~\ref{s:regularity}, and that of
Corollary~\ref{c:regGaussian} need the following basic estimates.
\begin{prop}\label{p:covEst}
Let $x$ be a random variable in $\MM$ with $\EE[d(\bmu, x)^2] <
\infty$.  Then for
$U,V,W \in \Smu$,
\begin{align*}
  \abs{\tau(x,V) - \tau(x,U)}
  &\leq
  \bigl(\EE  d(\bmu,x) + d(\bmu,x)\bigr) \dtan_\bmu(U, V),
\\
  \EE(\tau(x,V)-\tau(x,U))^2
  &\leq
  4(\EE d(\bmu, x)^2) \dtan_\bmu(U, V)^2,
\\\text{and}\quad
  \big|\tangCOV(x,U,V)- \tangCOV(x,U,W)\big|
  &\leq
  4 (\EE d(\bmu, x)^2)^{\frac12} \dtan_\bmu(W, V).
\end{align*}
\end{prop}
\begin{proof}
\edits{Since $\abs{\tau(x,V) - \tau(x,U)} \leq \abs{
\<\log_\bmu x, V\>_\bmu- \<\log_\bmu x, U\>_\bmu } + \abs{ m(x,V)-
m(x,U)}$, the first estimate follows because $|m(x,V) - m(x,U)| \leq
(\EE d(\bmu,x)) \dtan_\bmu(U, V)$ and $|\<X, V\>_\bmu - \<X, U\>_\bmu| \leq
d(\bmu,x)\dtan_\bmu(U, V)$ when $X = \log_\bmu x$.}  Squaring this
estimate, taking the expectation, and using the fact that $[\EE
d(\bmu, x)]^2 \leq \EE [ d(\bmu, x)^2]$ produces the second
inequality. \edits{ The last estimate follows from the first
two since $\|U\|=\|V\|=1$ and}
\begin{align*}
  |\tangCOV(U,V)- \tangCOV(U,W)|
  &\leq
  (\EE \tau(x,U)^2 \EE(\tau(x,V)-\tau(x,W))^2)^{\frac12}
\\\makebox[0pt][r]{and\hspace{.21\linewidth}}
  \EE \tau(x,U)^2
  &\leq 
  \edits{\EE (\<\log_\bmu x, U\>_\bmu-  m(x,U))^2\leq 4 \EE
  d(\bmu, x)^2 }.\qedhere
\end{align*}
\end{proof}

\begin{prop}\label{p:randomTangentHolder}
If $x$ is a random variable on $\MM$ with $\EE d(\bmu,x) < \infty$
then the centered random tangent field $V\mapsto \tau(x,V)$ induced by
$x$ is almost surely Lipchitz-continuous and hence $\tau(x,\;\cdot\;)
\in \cC(\Smu,\RR)$ almost surely.
\end{prop}
\begin{proof}
Proposition~\ref{p:covEst} yields the almost sure bound
\begin{align*}
  \abs{\tau(x,V) - \tau(x,U)}
  \leq
  \bigl(\EE d(\bmu,x) + d(\bmu,x)\bigr) \dtan_\bmu(U, V).
\end{align*}
Thus $\tau(x,\ccdot)$ is Lipchitz-continuous and $\tau(x,\ccdot) \in
\cC(\Smu,\RR)$ almost surely.
\end{proof}

\subsection{H\"older-continuous Gaussian tangent fields}\label{b:Holder}

\begin{cor}\label{c:regGaussian}
Let $\nu$ be a measure on $\MM$ with $\int_\MM d(\bmu,y)^2\nu(dy) <
\infty$ and $\tangCOV$ the tangent covariance induced by~$\nu$.
Then there exists a version of the Gaussian tangent field induced
by~$\nu$ that is almost surely a H\"older continuous function from
$\Smu \to \RR$ for all H\"older exponents between $(0,1)$.
\end{cor}
\begin{proof}
Let $G$ be the Gaussian tangent field induced by~$\nu$.  By direct
calculation,
\begin{align*}
  \EE\bigl(G(U)-G(V)\bigr)^2 &= \tangCOV(\nu,U,U) + \tangCOV(\nu,V,V) - 2
                                \tangCOV(\nu,U,V)
                             = \EE\bigl(\tau(x,U) - \tau(x,V)\bigr)^2,
\end{align*}
where $x$ is $\MM$-valued
and distributed as~$\mu$.  By the middle estimate in
Proposition~\ref{p:covEst},
\begin{align*}
   \EE (G(U)-G(V))^2 \leq K_2(\nu)  \dtan_\bmu(U, V)^2
\end{align*}
for some positive constant $K_2(\nu)$ depending on~$\nu$.  However,
since $G(U)-G(V)$ is \edits{a centered} Gaussian \edits{its higher moments can
be bounded by powers of the second moment with the right homogeneity.} Thus for any positive integer~$p$ there
exists a constant~$K_p(\nu)$ so that
\begin{align*}
   \EE |G(U)-G(V)|^p \leq K_p(\nu) \dtan_\bmu(U, V)^p.
\end{align*}
With this estimate in hand, the result on the H\"older continuity
follows from~\eqref{t:kolmogorov-reg}.
\end{proof}

\begin{remark}\label{r:regGaussian}
By Corollary~\ref{c:regGaussian}, if $\gamma \in (0,1)$ then one can
choose a $\gamma$-H\"older continuous version
of a Gaussian tangent field which shares a tangent covariance with a
given probability measure~$\nu$ on~$\MM$ with $\int_{\MM}
d(\bmu,y)^2\nu(dy) < \infty$.  Therefore, when discussing such
Gaussian tangent fields, our convention is to consider only such
H\"older continuous versions, without further comment.
\end{remark}

\subsection{Chaining and covering: proof of Theorem~\ref{t:kolmogorov-reg}}
\label{b:chaining}

\noindent
The proof of Theorem~\ref{t:kolmogorov-reg} requires
ideas from chaining and covering estimates.  Optimal estimates are not
the goal, given that simpler results suffice here.  The proofs are
standard,
though the setting requires some adaptation.  See
\cite{kallenberg1997, talagrand-2014} for more general and optimal
results.

Let $\ve\text{-cov}$ be the collection of all $\ve$-covers
of $\Smu$.  Compactness of the unit tangent sphere implies
some finite collection
$\{V_i\}$
satisfies $\Smu \subseteq \bigcup_i B_\ve(V_i)$.  \edits{Define the
\emph{covering number}} $N(\ve)$ of~$\Smu$~by
$$
  N(\ve) = \inf_{D \in \frac\ve2\text{-cov}} |D|
$$
where $|D|$ is the cardinality of the set $D$.

Let $d \in(0, \infty)$ be such that 
\begin{align}\label{eq:dDef1}
  \limsup_{n\to\infty} \frac 1n \log_2 N(2^{-n}) \leq d.
\end{align}
When $\MM$ is some Euclidean space $\RR^{k+1}$ write $N_k(\ve)$ for 
the \edits{covering number} of the $k$-dimensional unit
sphere, and let $d_k \in(0,\infty)$ be such that 
\begin{align}\label{eq:dDef2}
  \limsup_{n\to\infty} \frac 1n \log_2 N_k(2^{-n}) \leq d_k. 
\end{align}

\begin{lemma}\label{l:d-is-finite}
For any stratified $\CAT(\kappa)$ space~$\MM$,
the number~$d$
in~\eqref{eq:dDef1} can be chosen to be finite.
\end{lemma}
\begin{proof}
Suppose that $\{R_i\}_{i=1}^{\ell}$ is the list of strata of $\MM$ and 
$r_i$ is the dimension of $R_i$.  Suppose that $\cB_i(\ve)$ is an 
$\ve$-cover of $R_i \cap \Smu$.  Then $\bigcup_{i=1}^n\cB_i(\ve)$
is an $\ve$-cover of $\Smu$.  Hence $d$ can be chosen such that $
d \leq \sum_{i=1}^{\ell}d_{r_i}$, which is finite. 
\end{proof}

\begin{proof}[Proof of Theorem~\ref{t:kolmogorov-reg}]
The arguments loosely follow those in
\cite[Theorem~3.23]{kallenberg1997}.

Fix a sufficiently large, positive integer~$M$ and let $d$ be the constant from \eqref{eq:dDef2}.  For each $n \geq M$, let $\widehat D_n$
be the centers of the balls in a $2^{-(n+1)}$-cover such that
$|\widehat D_n| \leq \frac32 N(2^{-(n+1)}) \leq 2\cdot 2^{d(n+1)}$.  By definition of
$N(2^{-(n+1)})$ and $d$ from
\eqref{eq:dDef1} and \eqref{eq:dDef2}, this is always possible for $M$ large enough. Now
define $D_n = \bigcup_{k=M}^n \widehat D_n$ and observe that
$D_n \subset D_{n+1}$, and
$|D_n| < \mathcal{N}_n=\sum_{k=M}^n  2\cdot 2^{d(n+1)}\leq K
2^{d n}$ for some $K$ depending on $d$ but not $n$.
Set
$$
  \xi_k
  =
  \sup\bigl\{|f(V) - f(U)|
             \,\bigm|\,
             U,V \in D_k\text{ and }\dtan_\bmu(U, V) \leq 2^{-k}\bigr\}.
$$
For any integer  $a > \max(d,1)$, observe that
$$
  \xi_k^a\
  \leq
  \sum_{\substack{U,V \in D_k \\ \dtan_\bmu(U, V) \leq 2^{-k} }} |f(V) - f(U)|^a\,.
$$
By assumption $\EE |f(V) - f(U)|^a \leq C_a \dtan_\bmu(U, V)^a$
for some positive constant $C_a$ independent of $U$ and~$V$.  Since
$\bigl|\bigl\{U,V \in D_k \,\bigm|\, \dtan_\bmu(U, V) \leq
2^{-k}\bigr\}\bigr| \leq K 2^{dk}$,
$$
  \EE \xi_{k}^a \leq K2^{-ka}2^{dk} = K 2^{-k(a-d)}  
$$
for some possibly different constant~$K$.  Therefore, for any $\gamma
\in (0,(a-d)/a)$ and $a > d$,
$$
  \sum_k\EE \xi_{k}^a 2^{\gamma ka} \leq C \sum_k 2^{-k(a -a\gamma - d)}. 
$$
Since $a -a \gamma - d > 0$ this sum converges. \edits{Hence for some random but finite $K_0$, $\sum_k \xi_{k}^a 2^{\gamma ka}
  \leq K_0
$ almost surely.}
Consequently, for some random constant~$K$ and any $m \geq M$,
$$
  \sup\bigl\{|f(U)-f(V)|
             \,\bigm|\,
             U,V \in \bigcup_k D_k, \dtan_\bmu(V, U) \leq 2^{-m}\bigr\}
  \leq
  \Bigl(\sum_{k \geq m} \xi_k^a\Bigr)^{\frac 1a}
  \leq 
  K 2^{-m\gamma}
$$
almost surely, so $f$ is almost surely H\"older continuous with
exponent~$\gamma$ on the set $\bigcup_k D_k$.  Picking $a$ large
enough makes $(a-d)/a$ arbitrarily close to~$1$, so $\gamma$ can be
chosen arbitrarily close to~$1$.

Since $\bigcup_k D_k$ is dense in $\Smu$, a H\"older continuous
version $\hat f$ of~$f$ can be defined by setting $\hat f(U) = f(U)$
for $U \in \bigcup_k D_k$ and otherwise setting $\hat f(U) = \lim
f(U_k)$ for any sequence $U_k \in \bigcup_k D_k$ with $U_k \rightarrow
U$.  Since $\hat f(U_k)=f(U_k)$ is continuous on $\bigcup_k D_k$ this
definition is independent of the chosen sequence $U_k \rightarrow U$.
The resulting $\hat f$ is H\"older continuous for $\gamma <1$.

To see that $\hat f$ is a version of the field~$f$, \eqref{eq:ExpCont}
implies that $f(U_k)$ converges to $f(U)$ in probability while $ \hat
f(U_k)$ converges to $\hat f(U)$ by continuity of~$\hat f$ inherent in
its construction.  Hence $\bP\bigr(f(U) = \hat f(U)\bigr) = 1$ for all
$U \in \Smu$ since $f(U_k)=\hat f(U_k)$.
\end{proof}

\section{CLT for random tangent fields}\label{s:random-tangent}

\begin{defn}\label{d:empirical-tangent-field}\rm
Fix a measure~$\mu$ on a $\CAT(\kappa)$ space~$\MM$ and a sequence
$\xx = (x_1,x_2,\ldots)$ in~$\MM^\NN$ of mutually independent random
variables $x_i \sim \mu$.
Denote by $g_i(V) = \tau(x_i,V)$ the centered random tangent field
induced by~$\mu$ (Definition~\ref{d:tangent-field})\edits{,
where $\bmu$ there is the Fr\'echet mean of~$\mu$}.  Then
\begin{align*}
  \Gn(V) &= \frac 1{\sqrt n} \sum_{k=1}^n g_i(V)
\intertext{is the \emph{empirical tangent field} induced by~$\mu$.  In
contrast, the average of $g_1,\dots,g_n$ is}
   \bgn(V) &= \frac 1n \sum_{k=1}^n g_i(V)
           = \frac 1{\sqrt n} \Gn. \notag
\end{align*}
\end{defn}

The $g_i$ are centered, so their expectations vanish.  Their averages
$\bgn(V)$ have the classical Law of Large Numbers scaling and hence
converge to~$0$, while the $\Gn$ have the classical CLT scaling
selected to stablize the fluctuations.  The following central limit
theorem for the tangent field~$\Gn$ makes this precise, intrinsically
capturing the asymptotics of sample means in the linear, albeit
nontraditional setting of Gaussian fields.  The proof occupies
Section~\ref{s:proof-of-G-CLT}.


\begin{theorem}[Random tangent field CLT]\label{t:tangent-field-CLT}
Fix a localized measure~$\mu$ on a stratified $\CAT(\kappa)$
space~$\MM$.  Let $G$ be the Gaussian tangent field induced by~$\mu$
(Definition~\ref{d:gaussian-tangent-field}), and let $G_n = \sqrt
n\,\bgn$ be the empirical tangent fields
(Definition~\ref{d:empirical-tangent-field}) for the collection of
independent random variables~$x_i$ each distributed as~$\mu$.  Then
the $\Gn$ converge to~$G$
$$
  \Gn \overset{d\;}\to G
$$
in distribution as $n \rightarrow \infty$.
\end{theorem}

The following asserts that the limiting Gaussian tangent field $G$ and
the empirical tangent fields $G_n$ can be constructed on the same
probability space so that the convergence is almost sure as functions
in $\cC(\Smu, \RR)$.  Its proof is at the end of
Section~\ref{s:proof-of-G-CLT}.

\begin{cor}\label{c:cont-realization}
In the setting of Theorem~\ref{t:tangent-field-CLT}, there exist
versions of $\Gn$ and $G$ so that $\Gn \to G$ almost surely in
$\cC(\Smu, \RR)$ equipped with the supremum norm and so that $G$ is
almost surely~H\"older continuous for any H\"older exponent less
than~$1$.
\end{cor}

\section{Proof of the CLT for random tangent fields}\label{s:proof-of-G-CLT}

\noindent
This subsection culminates in the proofs of
Theorem~\ref{t:tangent-field-CLT} and
Corollary~\ref{c:cont-realization}.  It begins by collecting a number
of auxiliary results and observations.

Throughout the subsection, assume the situation of
Definition~\ref{d:empirical-tangent-field}.  For notational
convenience in this setting, write
\begin{align*}
  X_i = \log_\bmu x_i,
\qquad\text{so}\qquad
  X_i \sim \hmu = (\log_\bmu)_\sharp\mu
\end{align*}
as in
Definition~\ref{d:localized}.  Also use $\|V\| =
d(\bmu,\exp_\bmu V)$ to denote the length of a vector $V \in
\Tmu$.  Lastly, write $\tangCOV(V,U)$ for the covariance
$\tangCOV(\mu,V,U)$ of $\mu$ and $\Fmu = F_\mu$ for the Fr\'echet
function of~$\mu$ in~\eqref{eq:FrechetF}, as usual.

\subsection{A priori and martingale bounds on \texorpdfstring{$G_n$}{Gn}}
\label{b:useful-bounds}

\begin{lemma}\label{l:finite-second-moment}
The Gaussian random tangent field $G$ can be understood as a random
variable taking values in $L^2(\Smu,\RR)$ with finite second moment.
\end{lemma}
\begin{proof}
As the
\edits{angular pairing at~$\bmu$ is continuous
\cite[\!Lemma\,1.21]{shadow-geom}  so is the covariance
$\tangCOV(V,V)$. Since $\Smu$ is \mbox{compact}, we  know
that $\tangCOV(V,V)$ is bounded on $\Smu$ and in particular}
\begin{align*}
\hspace{.29692\linewidth}
  \EE\bigl(\norm{G}^2\bigr)
  :=
  \int_{\Smu} \tangCOV(V,V)dV
  <
  \infty.
\hspace{.29692\linewidth}
\qedhere
\end{align*}
\end{proof}

\begin{lemma}\label{l:almost_sure_cont-of_Gn}
For each $n \geq 1$, the empirical tangent field~$G_n$ is almost
surely Lipchitz-continuous.  In particular, $\Gn \in
\cC(\Smu,\RR)$ with probability~$1$.
\end{lemma}
\begin{proof}
There is an almost sure bound
\begin{align*}
  \abs{\Gn (V) - \Gn(U)}
  \leq
  \Bigl(\sqrt n \EE \norm{X_1} + \frac1{\sqrt n} \sum_{i=1}^n
  \norm{X_i}\Bigr) \dtan_\bmu(U, V)
\end{align*}
since $|m(x_i,\!V) - m(x_i,\!U)| \!\leq\! (\EE \norm{X_i})
\dtan_\bmu(U, \!V)$ and $|\<X_i, \!V\>_\bmu - \<X_i, \!U\>_\bmu|
\!\leq\! \norm{X_i} \dtan_\bmu(U, \!V)$.  Thus $\Gn$ is
Lipchitz-continuous and $\Gn \in \cC(\Smu,\RR)$ almost surely.
\end{proof}

\begin{lemma}\label{l:C_p}
For any $p \in \NN$, there exists a universal constant $C_p$ so that
$$
 \sup_{n \geq 1} \,\EE\:|\Gn(V)|^{2p} \leq C_{2p}{\gamma_{2p}}
$$
where $\gamma_{2p} = \EE\:\norm{X_1}^{2p}$.
\end{lemma}
\begin{proof}
The bound for $p > 2$ follows by reasoning as in the $p = 2$ case,
which follows~from
\begin{align*}
\EE\:\abs{\Gn(V)}^4
  & = \frac1{n^2}\sum_{\substack{i,j=1\\i \neq j}}^n \EE g_i(V)^2
      \EE g_j(V)^2 + \frac1{n^2} \sum_{i=1}^n\EE g_i(V)^4
\\
  &\leq
   \frac{n(n-1)\gamma_2^2}{n^2} + \frac{n\gamma_4}{n^2}
\\
  &\leq
   C_4 \gamma_4.\qedhere
\end{align*}
\end{proof}

Let $\cF_n$ be the $\sigma$-algebra generated by~$G_j(V)$ such that $j
\leq n$ and $V \in \Tmu$.

\begin{prop}\label{p:GnMart}
For fixed $V \in \Smu$, the assignment $n \mapsto \Gn(V)$ is a
martingale relative to the filtration $\cF_n$.  In fact, any fixed
$V_i \in \Tmu$ and $c_i \in \RR$ yield~a~martingale
$$
  n \mapsto \sum_{i=1}^k c_i \Gn(V_i).
$$
\end{prop}
\begin{proof}
By construction $G_n(V)$ is adapted to $\cF_n$ and $\sup_n \EE
|G_n(V)| < \infty$ by the estimates in Lemma~\ref{l:C_p}.
$\EE(G_{n+k}(V) \mid \cF_n) = \Gn(V)$ because each $g_i(V)$ has mean
zero and is independent of the random variables $\{g_k(V) \mid k \neq
i\}$.  The claim about finite linear combinations follows by
the same reasoning.  For more on the definition of a martingale and
further details see \cite{williams1991}.
\end{proof}

\begin{cor}\label{c:C_pSup}
For any $V 
\in \Smu$, 
$$
  \bP\Bigl(\sup_n \abs{\Gn(V)} > \beta\Bigr) 
  \leq 
  \frac{C_{2p}\gamma_{2p}}{\beta^{2p}}. 
$$
\end{cor}
\begin{proof}
The estimate follows from the classical Doob martingale bound and the
moment bounds in Lemma~\ref{l:C_p}; see \cite[Section~4.4]{Dur19} or
\cite{williams1991}, for example.
\end{proof}

\begin{lemma}\label{l:DiffMoments}
Given $p \in \NN$ there is a universal constant $C_p$ so that for any
$V, U \!\in\nolinebreak\! \Smu$,
$$
  \EE\:\abs{\Gn(V)- \Gn(U)}^{2p}
  \leq
  C_{2p}{(1 + \gamma_{2p})} \dtan_\bmu(U, V)^{2p},
$$
where $\gamma_{2p} = \EE\:\|X_1\|^{2p}$.  Additionally, for any $V,U
\in \Smu$,
$$
  \bP\Bigl(\sup_n \abs{\Gn(V)-\Gn(U)} > \beta\Bigr)
  \leq
  \frac{C_{2p}(1 + \gamma_{2p})}{\beta^{2p}}\dtan_\bmu(U, V)^{2p}
$$
for any $\beta > 0$.
\end{lemma}
\begin{proof}
As before, the proofs for higher moments follow the same pattern as
the estimate for $p = 2$, which is proved as follows.  Begin by
observing that
\begin{align*}
\EE\:\abs{g_i(V) - g_i(U)}^2
   &\leq 2\abs{m(V) - m(U)}^2 + 2\EE\:\abs{\<\log X_i, V\>_\bmu - \<\log X_i, U\>_\bmu}^2
\\*&\leq 2\abs{m(V) - m(U)}^2 + 2\EE\:\abs{\log X_i}^2 \dtan_\bmu(V, U)^2
\\*&\leq 2\bigl(1 + \EE\:\abs{\log X_i}^2\bigr) \dtan_\bmu(V, U)^2,
\end{align*}
so
$$
  \EE\:\abs{g_i(V) - g_i(U)}^4
  \leq C
  \bigl(1 + \EE\:\abs{\log X_i}^4\bigr) \dtan_\bmu(V, U)^4.
$$
Again, using the independent increments produces
\begin{align}
\notag
\EE\:\abs{\Gn (V) - \Gn(U)}^4
   &\leq \frac1{n^2}\! \sum_{\substack{j,i=1\\i \neq j}}^n
         \EE\:|g_i(V)-g_i(U)|^2 \, \EE\:|g_j(V)-g_j(U)|^2
\\\notag&\phantom{\leq}\
      + \frac1{n^2}\! \sum_{i=1}^n \EE\:|g_i(V)-g_i(U)|^4
\\*&\leq C (1 + \gamma_4) \dtan_\bmu(V, U)^4.
\label{eq:bound}
\end{align}
To get the final estimate from~\eqref{eq:bound}, apply Doob's
Martingale inequality, as in the proof of Corollary~\ref{c:C_pSup}, to
$n \mapsto \Gn(V) - \Gn(U)$, which is a martingale by
Proposition~\ref{p:GnMart}.
\end{proof}

\subsection{Convergence}\label{b:convergence}

\noindent
Begin with the following result on finite-dimensional distributions.

\begin{lemma}\label{l:finiteDim}
For any $k \in \NN$ and
vectors $V_1,\dots,V_k \in \Smu$ in the unit tangent sphere, the
vector $\xi_n = \bigl(\Gn(V_1), \dots,\Gn(V_k)\bigr) \in \RR^m$
converges in distribution to the Gaussian vector $\xi =
\bigl(G(V_1),\dots,G(V_k)\bigr)$ with covariance matrix $\tangCOV_{ij}
= \{\tangCOV(V_i,V_j) \mid 1 \leq i,j \leq k\}$.
\end{lemma}
\begin{proof}
The increments $\xi_n - \xi_{n-1}$ are independent, so the result
follows from the standard central limit theorem for independent random
variables on~$\RR^m$ after noting that the~$\xi_n$ share the same
covariance matrix~$\tangCOV$.
\end{proof}


The next lemma uses the notion of tightness (see
\cite[Section~5.1]{Bil13} for example): a family of measures on a
metric space~$S$ with the Borel $\sigma$-field~$\mathcal{S}$ is
\emph{tight} if for every $\ve > 0$
some compact set $K \subseteq S$
has $\nu(K) > 1 - \ve$ for every measure $\nu$ in the family.
The interest in tightness stems from the fact that by Prokhorov's
Theorem a tight sequence of probability measures is relatively compact
in distribution and hence must have a subsequence which converges in
distribution.

When considering convergence in distribution on a space of continuous
functions (such as $\cC(\Smu, \RR)$, where $\Gn$ lies by
Lemma~\ref{l:almost_sure_cont-of_Gn}), the Arzel\`a--Ascoli Theorem
can prove tightness by controlling the modulus of continuity.  If $h
\in \cC(\Smu, \RR)$ and
$$
  w(h,r) 
  =
  \sup\bigl\{|h(V)-h(U)|
             \,\bigm|\,
             U,V\in \Smu\text{ and }\dtan_\bmu(V, U) \leq r\bigr\},
$$
$$
  K(c_0,c_1,r_0)
  = \bigl\{h \in \cC(\Smu, \RR) \mid |h(V)| \leq c_0 \text{ for
    some } V \in \Smu \text{ and } w(h,r_0) \leq c_1 \bigr\}
$$
is relatively compact for any positive $c_0$, $c_1$, and~$r_0$.  By
construction, the collection of functions in $K(c_0,c_1,r_0)$ is
equicontinuous, so any limit point in the supremum topology is also
continuous.

\begin{lemma}\label{l:tight}
If $\MM$ is stratified then the sequence $\{\Gn\}$ is tight in
$\cC(\Smu, \RR)$.
\end{lemma}
\begin{proof}
The arguments are
similar to
\cite[Theorem~16.9]{kallenberg1997} and more generally
\cite{talagrand-2014}.  Fix $\MM$ and the collection
$\{D_n\}_{n=M}^\infty$ of covering ball centers as
for Theorem~\ref{t:kolmogorov-reg}, and~set
\begin{equation*}
  \xi_{n,k}
  =
  \sup\bigl\{|\Gn(V) - \Gn(U)|
             \,\bigm|\, 
             U,V \in D_k\text{ and }\dtan_\bmu (U, V) \leq 2^{-k}\bigr\}. 
\end{equation*}
For any even integer  $a > \max(d,1)$ where $d$ is the constant from \eqref{eq:dDef2}, observe that
\begin{align*}
   \xi_{n,k}^a \leq \sum_{ \substack{U,V \in D_k \\ \dtan_\bmu (U, V) \leq
  2^{-k} }} |\Gn(V) - \Gn(U)|^a
\end{align*}
and, by Lemma~\ref{l:DiffMoments}, $\EE\:|\Gn(V) - \Gn(U)|^a \leq C_a
\dtan_\bmu(U, V)^a$ for some positive constant $C_a$ independent of
$U$ and~$V$.  Since the set $\bigl\{U,V \in D_k \,\bigm|\, \dtan_\bmu
(U, V) \leq 2^{-k}\bigr\}$ has size $\leq K 2^{dk}$,
\begin{align}\label{eq:xiBound}
  \EE\:\xi_{n,k}^a \leq K2^{-ka}2^{dk} = K 2^{-k(a-d)}
\end{align}
for some possibly different constant $K$.  Next,
given any $U,V \in \Smu$ that satisfies $\dtan_\bmu (U, V) \leq
2^{-m}$ there exist two sequences $\{U_j\}_{j=m}^\infty$ and
$\{V_j\}_{j=m}^\infty$ with $U_k,V_k \in D_k$ and $\lim U_k = U$ and
$\lim_k V_k =V$.  Then
\begin{align*}
  |\Gn(U) -\Gn(V)|
  \leq
  |\Gn(U_m) -\Gn(V_m)| &+ \sum_{k=m}^\infty |\Gn(U_m) -\Gn(U_{m+1})|
\\                     &+ \sum_{k=m}^\infty |\Gn(V_m) -\Gn(V_{m+1})|
\end{align*}
since $\Gn$ is continuous.  This in turn implies that $|\Gn(U)
-\Gn(V)| \leq C \sum_{k \geq m} \xi_{n,k}$ for some universal positive
$C$, and hence $w(h,2^{-m}) \leq C \sum_{k \geq m} \xi_{n,k}$.  Since
$a > 1$, applying the triangle inequality to this estimate and then
\eqref{eq:xiBound} produces
\begin{align*}
  \EE\bigl(w(\Gn, 2^{-m})^a\bigr)
  \leq
  \Bigl(C \sum_{k \geq m} \bigl(\EE \xi_{n,k}^a\bigr)^{\frac 1a}\Bigr)^a
  \leq
  C \Bigl(\sum_{k \geq m} 2^{-k\bigl(1-\frac da\bigr)}\Bigr)^a
  \leq
  C 2^{-m(a - d)}.
\end{align*}
Since $a - d > 0$, this implies that
$$
  \lim_{m\to\infty} \limsup_{n\to\infty} \EE\bigl[w(\Gn, 2^{-m})\wedge 1\bigr] = 0,
$$
which shows the tightness of~$\{\Gn\}$.
\end{proof}

\begin{lemma}\label{l:functional-CLT}
If $\MM$ is stratified then the sequence $\Gn$ converges in
distribution to $G$.
\end{lemma}
\begin{proof}
The result follows from \cite[Theorem~7.5]{Bil13} or
\cite[Theorem~16.5]{kallenberg1997} in the setting of continuous
functions on~$\RR^d$.  We sketch the proof for completeness and to
translate it to our setting on a stratified space~$\MM$ rather
than~$\RR^d$.  The result follows from the classical result that $\Gn$
converges in distribution to $G$ if and only if every subsequence in
$\{\wh G_k = G_{n_k}\}_k$ contains a further subsequence which $\{\wh
G_{k_j}\}$ which converges in distribution to~$G$ (see
\cite[Theorem~2.6]{Bil13}, for example).  Since the $\{\Gn\}$ are
tight by Lemma~\ref{l:tight}, Prokhorov's Theorem in the setting of
continuous functions \cite[Theorem~7.1]{Bil13} implies that any
subsequence $\{\wh G_k = G_{n_k}\}_k$ contains a further subsequence
$\{\wh G_{k_j}\}$ which converges in distribution to a limit which we
denote by $\wh G \in \cC(\Smu, \RR)$.  Since $\wh G$ is the uniform
limit of continuous functions it is also continuous.  Since both $\wh
G$ and~$G$ are continuous, they are determined completely by their
finite-dimensional distributions.  However, since $\bigl(\wh
G_{k_j}(V_1), \dots, \wh G_{k_j}(V_n)\bigr)$ converges to
$\bigl(G(V_1),\dots, G(V_n)\bigr)$ for any collection $V_1,\dots,V_n
\in \Smu$, it must be that $G$ and $\wh G$ have the same distribution.
\end{proof}

\begin{proof}[Proof of Corollary~\ref{c:cont-realization}]
$\cC(\Smu,\RR)$ is a complete separable metric space when endowed with
the supremum norm.  Convergence then follows from
\cite[Theorem~3.30]{kallenberg1997}.  As each $\Gn$ is continuous and
$G$ is their uniform limit, $G$ is almost-sure continuous.
\end{proof}



\vspace{-1ex}


\end{document}